\documentclass{amsart}
\usepackage{amsfonts}
\usepackage{amssymb}
\usepackage{times}
\usepackage{xcolor}
\usepackage{amsmath}
\usepackage{amsthm}
\usepackage{comment}
\usepackage{xcolor}
\usepackage{graphicx}

\makeatletter
\def\Ddots{\mathinner{\mkern1mu\raise\p@
\vbox{\kern7\p@\hbox{.}}\mkern2mu
\raise4\p@\hbox{.}\mkern2mu\raise7\p@\hbox{.}\mkern1mu}}
\makeatother

\newtheorem{theorem}{Theorem}[section]

\newtheorem{example}{Example}[section]
\newtheorem{remark}{Remark}[section]

\newtheorem{question}{Question}[section]
\newenvironment{dem1,1}[1][Proof of Theorem \ref{mainthm1,1}]{\noindent\textit{#1.} }{\hfill $\square$}
\newenvironment{dem1,2}[1][Proof of Theorem \ref{mainthm1,2}]{\noindent\textit{#1.} }{\hfill $\square$}
\newenvironment{dem2,1}[1][Proof of Theorem \ref{mainthm2,1}]{\noindent\textit{#1.} }{\hfill $\square$}
\newenvironment{dem2,2}[1][Proof of Theorem \ref{mainthm2,2}]{\noindent\textit{#1.} }{\hfill $\square$}

\begin{document}

\title{Exploring Functional Identities: From Division Rings to Matrix Algebras}

\thanks{The first author was supported by FAPESP number 2022/14579-0.}

\author[Daniel Kawai]{Daniel Kawai}
\address{Daniel Kawai, S\~{a}o Paulo University, 
Rua Mat\~{a}o, 1010, 
CEP 05508-090, S\~{a}o Paulo, Brazil\\
{\em E-mail}: {\tt daniel.kawai@usp.br}}{}

\author[Bruno Leonardo Macedo Ferreira]{Bruno Leonardo Macedo Ferreira}
\address{Bruno Leonardo Macedo Ferreira, Federal University of Technology-Paran\'{a}-Brazil, Federal University of ABC and S\~{a}o Paulo University, 
Avenida Professora Laura Pacheco Bastos, 800, 
85053-510, Guarapuava, Brazil, Av. dos Estados, 5001 - Bangú, Santo André - SP, 09280-560 and Rua Mat\~{a}o, 1010, 
CEP 05508-090, S\~{a}o Paulo, Brazil \\\linebreak
{\em E-mail}: {\tt 
 brunolmfalg@gmail.com}}{}

\begin{abstract}
In this paper, we tackle unresolved inquiries by Ferreira et al. \cite{bruno} in their recent publication, ``Functional Identity on Division Algebras". We delve into the intricate behavior of additive functions on matrix algebras over division rings through rigorous analysis and theorem-proving. Our findings offer valuable insights into the nature of these functions and their implications for algebraic structures.\vspace*{.1cm}

\vspace*{.1cm}

\end{abstract}

\maketitle

\section{Introduction}
Over the decades, research in functional equations and identities has been pivotal in understanding the properties of mappings in algebraic structures and their relevance across various mathematical disciplines.

In a pioneering study in 1987, Vukman \cite{Vuk} studied a result on Cauchy's functional equations in division rings. He established conditions under which any additive mapping on a division ring with certain characteristics must be identically zero. He showed that if $D$ is a division ring with characteristics different from $2$ and $f : D \rightarrow D$ is an additive function satisfying:
\begin{equation}
\label{vukman}
f(x) = -x^2f(x^{-1})
\end{equation}
for all $x\neq 0$, then $f(x) = 0$ for all $x \in D$. This result, while fundamental, left additional questions to be explored regarding functional identities in more general algebraic structures.

Subsequently, Bre$\check{s}$ar \cite{bresar} expanded this line of inquiry by investigating functional identities involving two additive mappings, $F$ and $G$, in a division ring. He presented a characterization theorem, revealing specific forms that these mappings must adhere to, enriching our understanding of the structure of these mappings.

Motivated by Bre$\check{s}$ar's results and intending to extend these investigations further, this article sets out to resolve the questions raised by Catalano \cite{catalano} in 2018. Catalano explored functional identities in division rings, relating additive mappings to derivations and shedding light on their underlying properties.

In 2023, Ferreira et al. \cite{bruno} studied the functional equation:
\begin{equation}
\label{moraes}
x^{-1}f(x)+g(x^{-1})=0
\end{equation}
and proved the following theorem.
\begin{theorem}
Let $A=M_n(D)$ be the algebra of $n\times n$ square matrices over a division ring $D$ with characteristics different from $2$. Let $f,g: A \rightarrow A$ be additive functions, satisfying identity \eqref{moraes} for every invertible $x$. Then $f(x)=g(x)=0$, for all $x\in A$.
\end{theorem}
This work aims to address and resolve the open questions raised in the article \cite{bruno}, which delves into the investigation of functional identities in algebra. They posed the following two questions:
\begin{question}\label{q1}
Let $D$ be an algebra such that there exist additive functions $f,g : D \rightarrow D$ satisfying the following equation:
\begin{equation}
\label{vukman1}
x^{-n}f(x)+g(x^{-1})=0
\end{equation}
for every invertible $x$, where $n$ is an integer such that $n>2$. Is it possible to characterize the functions $f$ and $g$?
\end{question}
\begin{question}\label{q2}
If $D$ is a field of characteristic $2$ and there exist additive functions $f, g :D \rightarrow D$ satisfying identity \eqref{moraes} for every $x\neq0$, is $D$ a perfect field?
\end{question}

Note that Vukman's article deals with the particular case of equation \eqref{vukman1} where $n=2$ and $f=g$, and the main results of Ferreira et al. article \cite{bruno} deal with the particular case of equation \eqref{vukman1} where $n=1$.

The central focus of this work is to provide solutions to a modified version of Vukman's functional identity, specifically $x^{-1}f(x) + g(x^{-1}) = 0$, where $f$ and $g$ are additive mappings on a division algebra. Our approach aims to characterize the additive mappings that preserve the identity (\ref{vukman1}), exploring theoretical aspects.

Furthermore, we extend our investigation to matrix algebras over division algebras, examining analogous functional equations and providing characterizations for additive mappings that preserve these identities. We hope that this study will contribute not only to the understanding of functional equation theory but also to the advancement of knowledge in various areas of mathematics.

\section{Main Results}
In this section, we present our main theorems, which address open questions left by the recent article published by Ferreira et al. \cite{bruno}. 
Before presenting our primary results, it is pertinent to revisit the well-known identity attributed to Hua, an important element for the development of this paper.

Consider any two elements $a, b \in \mathcal{A}$, ensuring $ab \neq 0, 1$. We encounter the pivotal \textit{Hua's identity}:
\[a - \left[ a^{-1} + \left( b^{-1} - a \right)^{-1} \right]^{-1} = aba.\]

\begin{theorem}\label{t1}
Let $n$ be an odd integer and $m$ a positive integer. Let $D$ be a division ring with $\mathrm{Char}(D)\neq2$. Let $A=M_m(D)$ be the algebra of $m\times m$ matrices with entries in $D$. Let $A^*$ be the set of invertible matrices. Let $f,g:A\rightarrow A$ be additive functions such that:
\begin{equation}
\label{1}
x^{-n}f(x)+g(x^{-1})=0\quad\text{if $x\in A^*$}
\end{equation}
Then $f(x)=g(x)=0$ for all $x\in A$.
\end{theorem}

\begin{proof}
 We proceed by delineating several steps.

\textbf{Step 1:} To begin, we show that $f(x)=0$ for all $x\in A^*$. Let $x\in A^*$. Utilizing \eqref{1} with $x\mapsto x$ and $x\mapsto-x$, where $n$ is odd, yields:
\begin{equation*}
x^{-n}f(x)+g(x^{-1})=0,
\end{equation*}
\begin{equation*}
x^{-n}f(x)-g(x^{-1})=0,
\end{equation*}
implying $2x^{-n}f(x)=0$, and thus $f(x)=0$.

\textbf{Step 2:} Next, we establish that $g(x)=0$ for all $x\in A^*$. For $x\in A^*$, employing \eqref{1} with $x\mapsto x^{-1}$ yields:
\begin{equation*}
0=x^nf(x^{-1})+g(x)=g(x),
\end{equation*}
confirming $g(x)=0$.

\textbf{Step 3:} Consider $h\in\{f,g\}$. We aim to prove that $h(x)=0$ for all $x\in A$ such that $x^2=x$. For $x\in A$ satisfying $x^2=x$, observe:
\begin{equation*}
(1-2x)^2=1-4x+4x^2=1,
\end{equation*}
thus $1-2x\in A^*$. Additionally, $1\in A^*$, hence:
\begin{equation*}
0=h(1-2x)=-2h(x),
\end{equation*}
leading to $h(x)=0$.

\textbf{Step 4:} We conclude the proof. Let $h\in\{f,g\}$. For each $(i,j)$, let $e_{ij}\in A$ denote the matrix with the $(i,j)$ entry being $1$ and the others $0$. Then, for each $i$, $e_{ii}^2=e_{ii}$, implying $h(e_{ii})=0$.
\begin{itemize}
\item For $i$ and $a\in D$, if $a\neq0$, then $1+(a-1)e_{ii}\in A^*$, hence $h(1+(a-1)e_{ii})=0$, and consequently $h(ae_{ii})=0$. Furthermore, as $h$ is additive, $h(0e_{ii})=h(0)=0$.
\item For each $(i,j)$ with $i\neq j$ and for $a\in D$, $1+ae_{ij}\in A^*$, thus $h(1+ae_{ij})=0$, and hence $h(ae_{ij})=0$.
\end{itemize}
Let $x\in A$, then $x$ takes the form $x=\sum_{i,j}a_{ij}e_{ij}$ where $a_{ij}\in D$, yielding:
\begin{equation*}
h(x)=h\left(\sum_{i,j}a_{ij}e_{ij}\right)=\sum_{i,j}h\left(a_{ij}e_{ij}\right)=0
\end{equation*}
Therefore, $h(x)=0$ for all $x\in A$.
\end{proof}

And what about the scenario when $n$ is even? We have obtained fragmented results.

\begin{theorem}\label{t2}
Let $n$ be an even integer such that $n>2$. Let $D$ be a division ring with $\mathrm{Char}(D)>n$. Let $f,g:D\rightarrow D$ be additive functions such that:
\begin{equation}
\label{2}
x^{-n}f(x)+g(x^{-1})=0\quad\text{if $x\neq0$}
\end{equation}
Then $f(x)=g(x)=0$ for all $x\in D$.
\end{theorem}

\begin{proof}
Once again, we proceed with the proof of Theorem 4.1 by delineating several steps.

\medskip
\textbf{Step 1:} Let $x\neq0$. Then, utilizing \eqref{2} with $x\mapsto x$ and $x\mapsto x^{-1}$, we have:
\begin{equation}
\label{a}
x^{-n}f(x)+g(x^{-1})=0
\end{equation}
and also:
\begin{equation*}
x^nf(x^{-1})+g(x)=0
\end{equation*}
implying:
\begin{equation}
\label{b}
x^{-n}g(x)+f(x^{-1})=0
\end{equation}
Consequently, from \eqref{a} and \eqref{b}, we obtain:
\begin{equation*}
x^{-n}(f+g)(x)+(f+g)(x^{-1})=0
\end{equation*}
\begin{equation*}
x^{-n}(f-g)(x)-(f-g)(x^{-1})=0
\end{equation*}

\medskip
\textbf{Step 2:} Let $h\in\{f+g,f-g\}$. Taking $e=-1$ if $h=f+g$ and $e=1$ if $h=f-g$, we have:
\begin{equation}
\label{3}
h(x)=ex^nh(x^{-1})\quad\text{if $x\neq0$}
\end{equation}
Let $a\neq0$ and $a\neq1$, then by Hua's identity, we have:
\begin{equation}
a^2=a-(a^{-1}+(1-a)^{-1})^{-1}
\end{equation}
thus:
\begin{equation*}
\begin{array}{rcl}
h(a^2)&=&h(a-(a^{-1}+(1-a)^{-1})^{-1})\\
&=&h(a)-h((a^{-1}+(1-a)^{-1})^{-1})\\
&=&h(a)-e(a^{-1}+(1-a)^{-1})^{-n}h(a^{-1}+(1-a)^{-1})\\
&=&h(a)-ea^n(1-a)^nh(a^{-1}+(1-a)^{-1})\\
&=&h(a)-ea^n(1-a)^n(h(a^{-1})+h((1-a)^{-1}))\\
&=&h(a)-ea^n(1-a)^nh(a^{-1})-ea^n(1-a)^nh((1-a)^{-1})\\
&=&h(a)-e^2a^n(1-a)^na^{-n}h(a)-e^2a^n(1-a)^n(1-a)^{-n}h(1-a)\\
&=&h(a)-(1-a)^nh(a)-a^nh(1-a)\\
&=&h(a)-(1-a)^nh(a)-a^n(h(1)-h(a))\\
&=&h(a)-(1-a)^nh(a)-a^nh(1)+a^nh(a)\\
&=&(1+a^n-(1-a)^n)h(a)-a^nh(1),
\end{array}
\end{equation*}
thus:
\begin{equation*}
h(a^2)=(1+a^n-(1-a)^n)h(a)-a^nh(1).
\end{equation*}
The above equation also holds for $a=0$ and $a=1$ as well. As $n$ is even and $n\geq 4$, we may express this as:
\begin{equation*}
\left(h(a^2)-nah(a)\right)+\sum_{k=3}^{n-1}(-1)^{k-1}\binom{n}{k-1}a^{k-1}h(a)+\left(-na^{n-1}h(a)+a^nh(1)\right)=0
\end{equation*}
Now, let:
\begin{equation*}
\begin{array}{rcl}
t_0(x)&=&0\\
t_1(x)&=&0\\
t_2(x)&=&h(x^2)-nxh(x)\\
t_k(x)&=&(-1)^{k-1}\binom{n}{k-1}x^{k-1}h(x)\quad\text{se $3\leq k\leq n-1$}\\
t_n(x)&=&-nx^{n-1}h(x)+x^nh(1)
\end{array}
\end{equation*}
Each $t_k(x)$ is homogeneous with degree $k$ in $x$. Therefore, for all $a\in D$, we have:
\begin{equation}
\label{12}
t_0(a)+t_1(a)+\cdots+t_n(a)=0
\end{equation}
Let $a\in D$, then for all $m\in{0,\dots,n}$, by applying \eqref{12} with $ma$ instead of $a$, we derive:
\begin{equation*}
m^0t_0(a)+m^1t_1(a)+\cdots+m^nt_n(a)=0,
\end{equation*}
hence, if we set:
\begin{equation*}
V=\begin{pmatrix}
0^0&0^1&\cdots&0^n\\
1^0&1^1&\cdots&1^n\\
&&\vdots&\\
n^0&n^1&\cdots&n^n
\end{pmatrix},\quad\quad X=\begin{pmatrix}
t_0(a)\\
t_1(a)\\
\vdots\\
t_n(a)
\end{pmatrix},
\end{equation*}
then the matrix identity $VX=0$ holds. Thus, if $U$ is the adjoint\footnote{The adjoint matrix of $A$ is the transpose of the cofactor matrix of $A$.} of $V$, we have $UVX=0$, and since $UV=\det(V)I$, we get $\det(V)X=0$. By Vandermonde's determinant identity, we have:
\begin{equation*}
\det(V)=\prod_{0\leq i<j\leq n}(j-i),
\end{equation*}
but also $\mathrm{Char}(D)>n$, so $X=0$, particularly $t_2(a)=t_n(a)=0$, thereby yielding:
\begin{equation}
\label{5}
h(a^2)=nah(a)
\end{equation}
\begin{equation}
\label{6}
na^{n-1}h(a)=a^nh(1)
\end{equation}
Utilizing \eqref{6} with $1$ in place of $a$, we obtain $(n-1)h(1)=0$, yet $\mathrm{Char}(D)>n$, implying $h(1)=0$. Hence, for all $a\in D$ with $a\neq0$, by \eqref{6}, we have $na^{n-1}h(a)=a^nh(1)=0$, yet $\mathrm{Char}(D)>n$, hence $h(a)=0$. Moreover, as $h$ is additive, we have $h(0)=0$ as well.

\medskip
\textbf{Step 3:} Consequently, we conclude that $h(x)=0$ for all $x\in D$. Therefore, for all $x\in D$, we have $f(x)+g(x)=f(x)-g(x)=0$, implying $2f(x)=0$, hence $f(x)=0$ and also $g(x)=0$, completing the proof.
\end{proof}

\begin{remark}
It is essential to consider the conditions on the characteristic of the ring $D$. This necessity becomes evident in the following example.
\begin{example}\label{ex}
Let $n$ be any integer, and $p$ be a prime such that $p-1$ divides $n-2$. Consider $D=\mathbb{Z}_p$, the ring of integers modulo $p$, and the functions:
\begin{equation*}
f(x)=x \quad \text{and} \quad g(x)=-x.
\end{equation*}
Then $D$ is a division ring and $f,g: D\rightarrow D$ are additive functions satisfying equation \eqref{vukman1} for every $x\neq0$, for the Fermat's little theorem implies $x^{p-1}=1$ for every $x\neq0$. However, $f(1)\neq0$.
\end{example}
Example \ref{ex} illustrates the necessity of the condition $\mathrm{Char}(D) \neq 2$ in Theorem \ref{t1}, as for any integer $n$, the prime number $p = 2$ satisfies $p - 1 \mid n - 2$. Furthermore, it underscores the significance of retaining the condition $\mathrm{Char}(D)>n$ in Theorem \ref{t2}. A division ring always has a prime or infinite characteristic. Moreover, for even integers $n$ such that $n>2$, there may exist several prime numbers $p$ satisfying $p-1\mid n-2$. For instance, if $n=4$ or $n=6$, then every prime number $p$ such that $p\leq n$ satisfies $p-1\mid n-2$.
\end{remark}

Now, we endeavor to provide a negative response to Question \ref{q2}, illustrating that within the non-perfect field $D=\mathbb{Z}_2(t)$ of rational functions with coefficients in $\mathbb{Z}_2$, numerous solutions $(f,g)$ to the functional equation \eqref{moraes} exist, where $f$ and $g$ are additive.

\begin{theorem}
Let $D=\mathbb{Z}_2(t)$ denote the field of rational functions with coefficients in $\mathbb{Z}_2$. For any $A,B\in D$, there exists a unique pair $(f,g)$ consisting of additive functions $f,g:D\rightarrow D$ satisfying $f(1)=A$, $f(t)=B$, and equation \eqref{moraes} for every $x\in D$ such that $x\neq0$.
\end{theorem}
\begin{proof}
1) \textbf{Uniqueness}:
Let $f,g:D\rightarrow D$ be additive functions satisfying $f(1)=A$, $f(t)=B$, and equation \eqref{moraes} for all $x\neq0$. Then, for every $a,b\in D$ such that $ab\neq0,1$, using the Hua identity and the fact that $\mathrm{Char}(D)=2$, we have:
\begin{equation*}
aba=a+(a^{-1}+(b^{-1}+a)^{-1})^{-1}.
\end{equation*}
This yields:
\begin{equation*}
\begin{aligned}
f(aba) &= f(a+(a^{-1}+(b^{-1}+a)^{-1})^{-1}) \\
&= f(a)+f((a^{-1}+(b^{-1}+a)^{-1})^{-1}) \\
&\overset{\eqref{moraes}}{=} f(a)+(a^{-1}+(b^{-1}+a)^{-1})^{-1}g(a^{-1}+(b^{-1}+a)^{-1}) \\
&= f(a)+(aba+a)g(a^{-1}+(b^{-1}+a)^{-1}) \\
&= f(a)+(aba+a)(g(a^{-1})+g((b^{-1}+a)^{-1})) \\
&= f(a)+(aba+a)g(a^{-1})+(aba+a)g((b^{-1}+a)^{-1}) \\
&\overset{\eqref{moraes}}{=} f(a)+(aba+a)a^{-1}f(a)+(aba+a)(b^{-1}+a)^{-1}f(b^{-1}+a) \\
&= f(a)+(ab+1)f(a)+ab(a+b^{-1})(b^{-1}+a)^{-1}f(b^{-1}+a) \\
&= abf(a)+abf(b^{-1}+a) \\
&= abf(a)+ab(f(b^{-1})+f(a)) \\
&= abf(a)+abf(b^{-1})+abf(a) \\
&= abf(a)+abb^{-1}g(b)+abf(a) \\
&= abf(a)+ag(b)+abf(a) \\
&= ag(b),
\end{aligned}
\end{equation*}
which simplifies to:
\begin{equation}
\label{14}
f(aba)=ag(b).
\end{equation}

By $f(0)=g(0)=0$ and \eqref{moraes}, equation \eqref{14} also holds for $a,b\in D$ such that $ab=0$ or $ab=1$. Thus, \eqref{14} holds for all $a,b\in D$. By \eqref{14} with $a=1$, for all $b\in D$, we have $f(b)=g(b)$. So $f=g$, and for all $a\in D$, we have:
\begin{equation}
\label{15}
f(aba)=af(b).
\end{equation}

For every $n\in\mathbb{Z}$, by \eqref{15} with $a=t$ and $b=t^n$, we have:
\begin{equation}
\label{16}
f(t^{n+2})=tf(t^n).
\end{equation}

Therefore, by using induction on $n$ and \eqref{16}, for all $n\in\mathbb{Z}$, we have:
\begin{equation*}
f(t^{2n})=t^nA, \quad f(t^{2n+1})=t^nB.
\end{equation*}

Moreover, if $P(t)$ is a polynomial with coefficients in $\mathbb{Z}_2$ and $p\in\mathbb{Z}$, then it's evident that:
\begin{equation}
\label{18}
f(P(t^2)t^{2p})=P(t)t^{p}A, \quad f(P(t^2)t^{2p+1})=P(t)t^{p}B.
\end{equation}

Any element $x\in D$ can be represented as:
\begin{equation*}
x=\frac{P(t^2)+Q(t^2)t}{R(t^2)+S(t^2)t},
\end{equation*}
where $P(t),Q(t),R(t),S(t)$ are polynomials with coefficients in $\mathbb{Z}_2$. Hence, by representing:
\begin{equation*}
P(t)=\sum_{p=0}^{\hat{p}}P_pt^p, \quad Q(t)=\sum_{q=0}^{\hat{q}}Q_qt^q,
\end{equation*}
where $P_p,Q_q\in\mathbb{Z}_2$, we have:
\begin{equation*}
\begin{aligned}
f(x) &= f\left(\frac{P(t^2)+Q(t^2)t}{R(t^2)+S(t^2)t}\right) \\
&= \sum_{p=0}^{\hat{p}}P_pf\left(\frac{t^{2p}}{R(t^2)+S(t^2)t}\right)+\sum_{q=0}^{\hat{q}}Q_qf\left(\frac{t^{2q+1}}{R(t^2)+S(t^2)t}\right) \\
&\overset{\text{\eqref{moraes}}}{=} \sum_{p=0}^{\hat{p}}\frac{P_pt^{2p}}{R(t^2)+S(t^2)t}f\left(\frac{R(t^2)+S(t^2)t}{t^{2p}}\right)+\sum_{q=0}^{\hat{q}}\frac{Q_qt^{2q+1}}{R(t^2)+S(t^2)t}f\left(\frac{R(t^2)+S(t^2)t}{t^{2q+1}}\right) \\
&= \sum_{p=0}^{\hat{p}}\frac{P_pt^{2p}f(R(t^2)t^{-2p}+S(t^2)t^{-2p+1})}{R(t^2)+S(t^2)t}+\sum_{q=0}^{\hat{q}}\frac{Q_qt^{2q+1}f(R(t^2)t^{-2q-1}+S(t^2)t^{-2q})}{R(t^2)+S(t^2)t} \\
&\overset{\text{\eqref{18}}}{=} \sum_{p=0}^{\hat{p}}\frac{P_pt^p(R(t)A+S(t)B)}{R(t^2)+S(t^2)t}+\sum_{q=0}^{\hat{q}}\frac{Q_qt^q(R(t)B+S(t)tA)}{R(t^2)+S(t^2)t} \\
&= \frac{P(t)(R(t)A+S(t)B)}{R(t^2)+S(t^2)t}+\frac{Q(t)(R(t)B+S(t)tA)}{R(t^2)+S(t^2)t} \\
&= \frac{P(t)R(t)+Q(t)S(t)t}{R(t^2)+S(t^2)t}A+\frac{P(t)S(t)+Q(t)R(t)}{R(t^2)+S(t^2)t}B,
\end{aligned}
\end{equation*}
thus:
\begin{equation}
\label{17}
f\left(\frac{P(t^2)+Q(t^2)t}{R(t^2)+S(t^2)t}\right)=\frac{P(t)R(t)+Q(t)S(t)t}{R(t^2)+S(t^2)t}A+\frac{P(t)S(t)+Q(t)R(t)}{R(t^2)+S(t^2)t}B.
\end{equation}

2) \textbf{Existence}:
It remains to show that for any $A,B\in D$, there exists a function $f:D\rightarrow D$ such that \eqref{17} holds for any polynomials $P(t),Q(t),R(t),S(t)$ with coefficients in $\mathbb{Z}_2$, and to prove that indeed $f$ is additive, $f(1)=A$, $f(t)=B$, and if $g=f$, then $f$ and $g$ satisfy \eqref{moraes} for all $x\neq0$.

This task can be divided into steps. Notice that for every polynomial $P(t)$, we have:
\begin{equation}
\label{d3}
P(t^2)=P(t)^2,
\end{equation}
because the coefficients are in $\mathbb{Z}_2$.

\textbf{Step 1:} We will prove the existence of a function $f:D\rightarrow D$ satisfying \eqref{17} for any polynomials $P(t),Q(t),R(t),S(t)$ with coefficients in $\mathbb{Z}_2$. In other words, we will show that the function $f$ as described in \eqref{17} is well-defined.

Let $P(t),\hat{P}(t),Q(t),\hat{Q}(r),R(t),\hat{R}(t),S(t),\hat{S}(t)$ be polynomials such that:
\begin{equation*}
\frac{P(t^2)+Q(t^2)t}{R(t^2)+S(t^2)t}=\frac{\hat{P}(t^2)+\hat{Q}(t^2)t}{\hat{R}(t^2)+\hat{S}(t^2)t}.
\end{equation*}
Then:
\begin{equation*}
(P(t^2)+Q(t^2)t)(\hat{R}(t^2)+\hat{S}(t^2)t)=(\hat{P}(t^2)+\hat{Q}(t^2)t)(R(t^2)+S(t^2)t).
\end{equation*}
Hence:
\begin{equation*}
\begin{aligned}
&(P(t^2)\hat{R}(t^2)+Q(t^2)\hat{S}(t^2)t^2)+(P(t^2)\hat{S}(t^2)+Q(t^2)\hat{R}(t^2))t \\
&=(\hat{P}(t^2)R(t^2)+\hat{Q}(t^2)S(t^2)t^2)+(\hat{P}(t^2)S(t^2)+\hat{Q}(t^2)R(t^2))t.
\end{aligned}
\end{equation*}

Both sides of the above equality are polynomials over $t$, so we can compare their coefficients and obtain the following equalities:
\begin{equation}
\label{d1}
P(t)\hat{R}(t)+Q(t)\hat{S}(t)t=\hat{P}(t)R(t)+\hat{Q}(t)S(t)t
\end{equation}
\begin{equation}
\label{d2}
P(t)\hat{S}(t)+Q(t)\hat{R}(t)=\hat{P}(t)S(t)+\hat{Q}(t)R(t)
\end{equation}

Multiplying \eqref{d1} by $R(t)\hat{R}(t)$ and by $S(t)\hat{S}(t)t$, and multiplying \eqref{d2} by $R(t)\hat{S}(t)t$ and by $S(t)\hat{R}(t)t$, we obtain respectively:
\begin{equation*}
\begin{aligned}
&P(t)R(t)\hat{R}(t)^2+Q(t)R(t)\hat{R}(t)\hat{S}(t)t=\hat{P}(t)\hat{R}(t)R(t)^2+\hat{Q}(t)\hat{R}(t)R(t)S(t)t
\end{aligned}
\end{equation*}
\begin{equation*}
\begin{aligned}
&P(t)S(t)\hat{R}(t)\hat{S}(t)t+Q(t)S(t)\hat{S}(t)^2t^2=\hat{P}(t)\hat{S}(t)R(t)S(t)t+\hat{Q}(t)\hat{S}(t)S(t)^2t^2
\end{aligned}
\end{equation*}
\begin{equation*}
\begin{aligned}
&P(t)R(t)\hat{S}(t)^2t+Q(t)R(t)\hat{R}(t)\hat{S}(t)t=\hat{P}(t)\hat{S}(t)R(t)S(t)t+\hat{Q}(t)\hat{S}(t)R(t)^2t
\end{aligned}
\end{equation*}
\begin{equation*}
\begin{aligned}
&P(t)S(t)\hat{R}(t)\hat{S}(t)t+Q(t)S(t)\hat{R}(t)^2t=\hat{P}(t)\hat{R}(t)S(t)^2t+\hat{Q}(t)\hat{R}(t)R(t)S(t)t
\end{aligned}
\end{equation*}

By summing up the above four equations, and using $\mathrm{Char}(D)=2$ and \eqref{d3}, we obtain:
\begin{equation*}
(P(t)R(t)+Q(t)S(t)t)(\hat{R}(t^2)+\hat{S}(t^2)t)=(\hat{P}(t)\hat{R}(t)+\hat{Q}(t)\hat{S}(t)t)(R(t^2)+S(t^2)t)
\end{equation*}
so that:
\begin{equation}
\label{ea1}
\frac{P(t)R(t)+Q(t)S(t)t}{R(t^2)+S(t^2)t}=\frac{\hat{P}(t)\hat{R}(t)+\hat{Q}(t)\hat{S}(t)t}{\hat{R}(t^2)+\hat{S}(t^2)t}
\end{equation}

Now, multiplying \eqref{d1} by $S(t)\hat{R}(t)$ and by $R(t)\hat{S}(t)$, and multiplying \eqref{d2} by $S(t)\hat{S}(t)t$ and by $R(t)\hat{R}(t)$, we obtain respectively:
\begin{equation*}
\begin{aligned}
&P(t)S(t)\hat{R}(t)^2+Q(t)S(t)\hat{R}(t)\hat{S}(t)t=\hat{P}(t)\hat{R}(t)R(t)S(t)+\hat{Q}(t)\hat{R}(t)S(t)^2t
\end{aligned}
\end{equation*}
\begin{equation*}
\begin{aligned}
&P(t)R(t)\hat{R}(t)\hat{S}(t)+Q(t)R(t)\hat{S}(t)^2t=\hat{P}(t)\hat{S}(t)R(t)^2+\hat{Q}(t)\hat{S}(t)R(t)S(t)t
\end{aligned}
\end{equation*}
\begin{equation*}
\begin{aligned}
&P(t)S(t)\hat{S}(t)^2t+Q(t)S(t)\hat{R}(t)\hat{S}(t)t=\hat{P}(t)\hat{S}(t)S(t)^2t+\hat{Q}(t)\hat{S}(t)R(t)S(t)t
\end{aligned}
\end{equation*}
\begin{equation*}
\begin{aligned}
&P(t)R(t)\hat{R}(t)\hat{S}(t)+Q(t)R(t)\hat{R}(t)^2=\hat{P}(t)\hat{R}(t)R(t)S(t)+\hat{Q}(t)\hat{R}(t)R(t)^2
\end{aligned}
\end{equation*}

By adding the above four equations, and using $\mathrm{Char}(D)=2$ and \eqref{d3}, we have:
\begin{equation*}
(P(t)S(t)+Q(t)R(t))(\hat{R}(t^2)+\hat{S}(t^2)t)=(\hat{P}(t)\hat{S}(t)+\hat{Q}(t)\hat{R}(t))(R(t^2)+S(t^2)t)
\end{equation*}
so that:
\begin{equation}
\label{ea2}
\frac{P(t)S(t)+Q(t)R(t)}{R(t^2)+S(t^2)t}=\frac{\hat{P}(t)\hat{S}(t)+\hat{Q}(t)\hat{R}(t)}{\hat{R}(t^2)+\hat{S}(t^2)t}
\end{equation}

Multiplying \eqref{ea1} by $A$, multiplying \eqref{ea2} by $B$, and summing, we obtain:
\begin{equation*}
\frac{P(t)R(t)+Q(t)S(t)t}{R(t^2)+S(t^2)t}A+\frac{P(t)S(t)+Q(t)R(t)}{R(t^2)+S(t^2)t}B
\end{equation*}
\begin{equation*}
=\frac{\hat{P}(t)\hat{R}(t)+\hat{Q}(t)\hat{S}(t)t}{\hat{R}(t^2)+\hat{S}(t^2)t}A+\frac{\hat{P}(t)\hat{S}(t)+\hat{Q}(t)\hat{R}(t)}{\hat{R}(t^2)+\hat{S}(t^2)t}B
\end{equation*}

\textbf{Step 2:} Let's prove that $f$ is additive. Take $x$ and $y$ from $D$. By bringing the fractions to a common denominator, we can express them as polynomials $P(t)$, $\hat{P}(t)$, $Q(t)$, $\hat{Q}(t)$, $R(t)$, and $S(t)$ as follows:

\begin{equation*}
x=\frac{P(t^2)+Q(t^2)t}{R(t^2)+S(t^2)t},\quad\quad y=\frac{\hat{P}(t^2)+\hat{Q}(t^2)t}{R(t^2)+S(t^2)t},
\end{equation*}

resulting in:

\begin{equation*}
x+y=\frac{(P(t^2)+\hat{P}(t^2))+(Q(t^2)+\hat{Q}(t^2))t}{R(t^2)+S(t^2)t}.
\end{equation*}

Therefore,

\begin{equation*}
\begin{aligned}
f(x+y)&\overset{\eqref{17}}{=}\frac{(P(t)+\hat{P}(t))R(t)+(Q(t)+\hat{Q}(t))S(t)t)}{R(t^2)+S(t^2)t}A\\
&+\frac{(P(t)+\hat{P}(t))S(t)+(Q(t)+\hat{Q}(t))R(t)}{R(t^2)+S(t^2)t}B\\
&=\frac{P(t)R(t)+Q(t)S(t)t}{R(t^2)+S(t^2)t}A+\frac{P(t)S(t)+Q(t)R(t)}{R(t^2)+S(t^2)t}B\\
&+\frac{\hat{P}(t)R(t)+\hat{Q}(t)S(t)t}{R(t^2)+S(t^2)t}A+\frac{\hat{P}(t)S(t)+\hat{Q}(t)R(t)}{R(t^2)+S(t^2)t}B\\
&\overset{\eqref{17}}{=}f(x)+f(y).
\end{aligned}
\end{equation*}

\textbf{Step 3:} Let's establish that $f(1)=A$ and $f(t)=B$. Indeed, we have:

\begin{equation*}
1=\frac{1+0t}{1+0t},
\end{equation*}

thus,

\begin{equation*}
f(1)\overset{\eqref{17}}{=}\frac{1\cdot1+0\cdot0t}{1+0t}A+\frac{1\cdot 0+0\cdot 1}{1+0t}B=A.
\end{equation*}

Moreover:

\begin{equation*}
t=\frac{0+1t}{1+0t},
\end{equation*}

therefore,

\begin{equation*}
f(t)\overset{\eqref{17}}{=}\frac{0\cdot1+1\cdot0t}{1+0t}A+\frac{0\cdot 0+1\cdot 1}{1+0t}B=B.
\end{equation*}

\textbf{Step 4:} Let's show that if $g=f$, then $(f,g)$ satisfy \eqref{moraes} for all $x\neq0$. Let $x\neq0$, then we have polynomials $P(t)$, $Q(t)$, $R(t)$, and $S(t)$ such that:

\begin{equation*}
x=\frac{P(t^2)+Q(t^2)t}{R(t^2)+S(t^2)t},
\end{equation*}

hence,

\begin{equation*}
x^{-1}=\frac{R(t^2)+S(t^2)t}{P(t^2)+Q(t^2)t},
\end{equation*}

thus,

\begin{equation*}
\begin{aligned}
xf(x^{-1})&\overset{\eqref{17}}{=}\frac{P(t^2)+Q(t^2)t}{R(t^2)+S(t^2)t}\cdot\left(\frac{R(t)P(t)+S(t)Q(t)t}{P(t^2)+Q(t^2)t}A+\frac{R(t)Q(t)+S(t)P(t)}{P(t^2)+Q(t^2)t}B\right)\\
&=\frac{P(t)R(t)+Q(t)S(t)t}{R(t^2)+S(t^2)t}A+\frac{P(t)S(t)+Q(t)R(t)}{R(t^2)+S(t^2)t}B\\
&\overset{\eqref{17}}{=}f(x),
\end{aligned}
\end{equation*}

so, using $\mathrm{Char}(D)=2$, we find:

\begin{equation*}
x^{-1}f(x)+f(x^{-1})=0.
\end{equation*}

Hence, if $g=f$, then $(f,g)$ satisfy \eqref{moraes} for all $x\neq0$.

\textbf{Conclusion}:
Thus, given any $A,B\in D$, we have proven both uniqueness and existence of the additive functions $f,g:D\rightarrow D$ satisfying \eqref{moraes} for every $x\neq0$. Therefore, the theorem is established.
\end{proof}

\end{document}